\documentclass[12pt]{amsart}
\usepackage{latexsym, amsmath, amssymb, amsthm}
\usepackage{epsfig}
\usepackage{amscd}
\usepackage{latexsym}
\usepackage{pstricks,pst-node,cite}

\newtheorem{thm}{Theorem}
\newtheorem{lem}[thm]{Lemma}
\newtheorem{prop}[thm]{Proposition}
\newtheorem{cor}[thm]{Corollary}

\newtheorem{open}[thm]{Problem}

\def\scaleddraw #1 by #2 (#3 scaled #4){{
  \dimen0=#1 \dimen1=#2
  \divide\dimen0 by 1000 \multiply\dimen0 by #4
  \divide\dimen1 by 1000 \multiply\dimen1 by #4
  \draw \dimen0 by \dimen1 (#3 scaled #4)}
  }

\ifx\blackandwhite\undefined
  \psset{linecolor=blue}\newrgbcolor{darkred}{.9 0 0}\newrgbcolor{emgreen}{0 .9 0}
  \newrgbcolor{lightgreen}{.7  .99 .7}\newrgbcolor{lightred}{.99 .7 .7} \else
  \psset{linecolor=blue}\newgray{darkred}{.7}\newgray{emgreen}{0.3}\newgray{pup}{.5}\newgray{xxx}{.5}
\fi
 \psset{linewidth=1pt,dash=3pt 3pt,doublesep=.05,dotsize=1pt 5}
\SpecialCoor

\begin{document}

\title[Total domination vertex critical graphs]
{On the existence problem of the total domination vertex critical graphs}

\author{Moo Young Sohn}
\address{Mathematics, Changwon National
University Changwon 641-773, Korea}
\email{mysohn@changwon.ac.kr}
\thanks{This work was supported by Basic Science Research Program through the National Research Foundation of Korea(NRF) grant funded by the Korea government(MEST)(2009-0073714).
}

\author{Dongseok Kim}
\address{Department of Mathematics \\Kyonggi University
\\ Suwon, 443-760 Korea}
\email{dongseok@kgu.ac.kr}
\thanks{}

\author{Young Soo Kwon}
\address{Department of Mathematics \\Yeungnam University \\Kyongsan, 712-749, Korea}
\email{ysookwon@yu.ac.kr}
\thanks{}

\author{Jaeun Lee}
\address{Department of Mathematics \\Yeungnam University \\Kyongsan, 712-749, Korea}
\email{julee@yu.ac.kr}



\maketitle

\begin{abstract}
The existence problem of the total domination vertex critical graphs
has been studied in a series of articles.
The aim of the present article is twofold. First, we settle the
existence problem with respect to the parities of the total
domination number $m$ and the maximum degree $\Delta$ : for even
$m$ except $m=4$, there is no $m$-$\gamma_t$-critical graph
regardless of the parity of $\Delta$;  for $m=4$ or odd $m \ge 3$
and for even $\Delta$, an $m$-$\gamma_t$-critical graph exists if
and only if $\Delta \ge 2 \lfloor \frac{m-1}{2}\rfloor$; for $m=4$
or odd $m \ge 3$ and for odd $\Delta$, if $\Delta \ge 2\lfloor
\frac{m-1}{2}\rfloor +7$, then $m$-$\gamma_t$-critical graphs
exist, if  $\Delta < 2\lfloor \frac{m-1}{2}\rfloor$, then
$m$-$\gamma_t$-critical graphs do not exist. The only remaining
open cases are $\Delta = 2\lfloor \frac{m-1}{2}\rfloor +k$, $k=1,
3, 5$. Second, we study these remaining open cases when $m=4$ or
odd $m \ge 9$. As the previously known result for $m =
3$~\cite{GHHM, CS}, we also show that for $\Delta(G)= 3, 5, 7$,
there is no $4$-$\gamma_{t}$-critical graph of order
$\Delta(G)+4$. On the contrary, it is shown that for odd $m \ge 9$
there exists an $m$-$\gamma_t$-critical graph for all $\Delta \ge
m-1$.
\end{abstract}

\section{Introduction} \label{intr}

A domination and its variations in graph theory have been studied widely and extensively
because of its rich applications~\cite{GHHM, HM, HMM, MR}.
Two books by Haynes, Hedetniemi and Slater provide a well written survey on this subject~\cite{HHS, HHS2}.
We refer to~\cite{HHS} for notation and general terminology.

Let $G=(V(G),E(G))$ be a simple graph of order $n(G)$.
The minimum degree and the maximum degree of a graph $G$ are denoted by $\delta(G)$ and
$\Delta(G)$, respectively.
A subset $S\subseteq V$ is \emph{a dominating  set} of $G$ if every
vertex not in $S$ is adjacent to a vertex in $S$. The
\emph{domination number} of $G$, denoted by $\gamma(G)$, is the minimum
cardinality of dominating sets. A subset $S\subseteq V$ is \emph{a total dominating set} of $G$ if every vertex of $G$ is adjacent to
a vertex in $S$. The \emph{total domination number} of $G$, denoted
by $\gamma_{t}(G)$, is the minimum cardinality of total
dominating sets. A total dominating set of cardinality
$\gamma_{t}(G)$ is called a \emph{$\gamma_{t}(G)$-set}.

Goddard et al. introduced the concept of total domination critical
graphs~\cite{GHHM}. A graph $G$ with no isolated vertex is
\emph{total domination vertex critical} if for any vertex $v$ of
$G$ that is not adjacent to a leaf, a vertex of degree one, the
total domination number of $G-v$ is less than the total domination
number of $G$. Such a graph is said to be
\emph{$\gamma_{t}$-critical} or \emph{$m$-$\gamma_{t}$-critical}
if its total domination number is $m$.
It is well known that the order of \emph{$m$-$\gamma_{t}$}-critical graph $G$  is at least $\Delta(G)+m$. So, they suggested the following classification
problem of the total domination critical graphs.

\begin{open} [\cite{GHHM}]\label{open1}
Characterize $m$-$\gamma_{t}$-critical graphs $G$ with order
$\Delta(G)+m$.
\end{open}

There have been a series of articles regarding this problem. Mojdeh and Rad found $3$-$\gamma_{t}$-critical graphs of
order $3 + \Delta(G)$ for any even $\Delta(G)$ and showed that
there is no $3$-$\gamma_{t}$-critical graph $G$ of order $3 + \Delta(G)$ for
$\Delta(G)= 3, 5$~\cite{MR}. In~\cite{CS}, Chen and Sohn
proved that there is no $3$-$\gamma_{t}$-critical graph of order
$\Delta(G)+3$ with $\Delta(G)=7$ and $\delta(G)\geq 2$.
Furthermore, they gave a family of $3$-$\gamma_{t}$-critical
graphs of order $\Delta(G)+3$ with odd $\Delta(G)\geq 9$ and
$\delta(G)\geq 2$. Hassankhani and Rad proved that there is no $4$-$\gamma_{t}$-critical graph of
order $\Delta(G)+4$ with $\delta(G)\geq 2$ for $\Delta(G)= 3,
5$~\cite{HaR}. There have been several partial results on the
existence problem of the total domination vertex critical graphs
from different point of views.

The aim of the present article is twofold. First, we settle the
existence problem with respect to the parities of the total
domination number $m$ and the maximum degree $\Delta$ in Theorem~\ref{mainthm0}.

\begin{thm} \label{mainthm0} If there exists an $m$-$\gamma_t$-critical graph of order $\Delta+m$ for some $\Delta$ then $m=4$ or $m \ge 3$ is odd. Conversely, for any $m=4$ or odd $m \ge 3$,
\begin{enumerate}
\item[$(1)$] if  $\Delta < 2\lfloor\frac{m-1}{2}\rfloor$, then there exists no $m$-$\gamma_t$-critical graph of order $\Delta+m$.
\item[$(2)$]
 For any even $\Delta \ge 2\lfloor \frac{m-1}{2}\rfloor$, there exists an $m$-$\gamma_t$-critical graph of order $\Delta+m$.
\item[$(3)$] For any odd $\Delta \ge 2\lfloor \frac{m-1}{2}\rfloor
+7$, there exists an $m$-$\gamma_t$-critical graphs of order
$\Delta+m$.
 \end{enumerate}
\end{thm}

Theorem~\ref{mainthm0} implies that the only remaining cases are
$\Delta = 2\lfloor\frac{m-1}{2}\rfloor +k$, $k=1, 3, 5$. Second,
we study these remaining open cases when $m=4$ or odd $m \ge 9$. When $m=4$, we show that
there is a $4$-$\gamma_{t}$-critical graph $G$ of order $\Delta(G) +4$ with
$\delta(G) \geq 2$ if and only if  $\Delta(G) = 2,4,6,8$   or  $\Delta(G) \ge 9$.
For odd $m \ge 9$, it is shown that there exists an $m$-$\gamma_{t}$-critical graph $G$ of order $\Delta(G) +m$ with $\delta(G) \ge 2$ if and only if   $\Delta(G) \ge
m-1$.


The outline of this paper is as follows. In
section~\ref{preliminaries}, we review some definitions and
previous results. In section~\ref{intr}, some properties of
$m$-$\gamma_{t}$-critical graph of order $\Delta+m$ will be
given. In section~\ref{main}, we provide the proof of the
Theorem~\ref{mainthm0}. In section~\ref{m4}, we deal with the
remaining open cases for $m=4$ and $m \ge 9$.

 \section{Preliminaries} \label{preliminaries}

In this section, we review some definitions and previous results.
The degree, neighborhood and closed neighborhood of a vertex $v$ in a
graph $G$ are denoted by $d(v)$, $N(v)$ and $N[v]=N(v)\cup\{v\}$,
respectively. For a subset $S$ of $V$, we set $N(S)=\bigcup_{v\in S}N(v)$
and $N[S]=N(S)\cup S$. The graph induced by $S\subseteq V$ is denoted by $G[S]$.
The cycle, path and complete graph on $n$ vertices are
denoted by $C_{n}$,  $P_{n}$ and $K_{n}$, respectively. A vertex of degree one is called a \emph{leaf}.
A vertex $v$ of $G$ is called a \emph{support vertex} if it is adjacent to a
leaf. Let $S(G)$ be the set of all support vertices of $G$. The \emph{corona} of
a graph $H$, denoted by $cor(H)$, is the graph obtained from $H$ by adding a leaf adjacent
to each vertex of $H$.

For two graphs $G_1$ and $G_2$  and for two vertices $v_1 \in V(G_1)$ and $v_2 \in V(G_2)$, a \emph{vertex amalgamation} of $G_1$ and $G_2$ with two vertices $v_1$ and $v_2$ is a graph whose vertex set is $(V(G_1)-v_1)\cup (V(G_2)-v_2) \cup \{ v \}$ and edge set is \
\[ E(G_1-v_1) \cup E(G_2 - v_2) \cup \{ vu | v_1u \in E(G_1) \} \cup  \{ vw | v_2w \in E(G_2) \}. \]
The vertex amalgamation method is useful to construct a new
$\gamma_t$-critical graph by the following proposition.

\begin{prop}  [\cite{GHHM}] \label{vertex-amal-prop}
Let $F$ and $H$ be $j$-$\gamma_t$-critical and $k$-$\gamma_t$-critical graphs, respectively, with minimum degrees at least two and let $G$ be a graph formed by identifying a vertex of $F$ with a vertex of $H$. If $\gamma_t (G)=j+k-1$ then $G$ is also $\gamma_t$-critical.
\end{prop}

\begin{lem} \label{vertex-amal}
For any $i=1,2$, let $G_i$ be an $m_i$-$\gamma_t$-critical graph
$G_i$ of order $\Delta (G_i)+m_i$ with $\delta(G_i) \ge 2$ and let
$v_i \in V(G_i)$ be a vertex of maximum degree in $G_i$. If each
component of $G[V(G_i)-N[v_i]]$ is a $P_2$ then  the vertex
amalgamation $G$ of $G_1$ and $G_2$ with $v_1$ and $v_2$ is an
$(m_1+m_2-1)$-$\gamma_t$-critical graph  of order $\Delta
(G)+m_1+m_2-1$, where $\Delta(G) = \Delta(G_1)+\Delta(G_2)$.
\end{lem}

\begin{proof}
Let $v$ be the vertex of $G$ whose degree is $\Delta(G) = \Delta(G_1)+\Delta(G_2)$, namely, $v$ is an amalgamated vertex.
 For any $u \in N(v)$, $(V(G)-N[v])\cup \{u \}$ is a total dominating set of $G$ and whose cardinality is $m_1+m_2-1$. Hence $\gamma_t(G) \le m_1 + m_2 -1$. Let $S$ be a $\gamma_t(G)$-set of $G$.
Suppose $v \in S$. Then, $v$ is adjacent to a vertex $u \in S-\{ v \}$. Without loss of generality, we may assume that $u \in V(G_1)$. Then, $(V(G_1)\cap (S-\{ v \}) )\cup \{v_1 \})$ is a total dominating set of $G_1$. Furthermore, for $S$ to dominate $G_2 -N[v]$, we have $|V(G_2)\cap  (S-\{ v \})| \ge m_2-1$. Hence, $|S| \ge m_1 + m_2 -1$, which means that  $\gamma_t(G)=m_1+m_2-1$. By Proposition \ref{vertex-amal-prop}, $G$ is an
$(m_1+m_2-1)$-$\gamma_t$-critical graph  of order $\Delta
(G)+m_1+m_2-1$.
\end{proof}

The following two lemmas are known results in~\cite{GHHM} which
will be used in this paper.

\begin{lem} [\cite{GHHM}] \label{no-neighbor} If $G$ is a
$\gamma_{t}$-critical graph, then
$\gamma_{t}(G-v)=\gamma_{t}(G)-1$ for every $v\in V-S(G)$.
Furthermore, a $\gamma_{t}(G-v)$-set contains no neighbor of $v$. \label{lem1}
\end{lem}

\begin{lem} [~\cite{GHHM}] If a graph $G$ has
nonadjacent vertices $u$ and $v$ such that  $v\notin S(G)$ and $N(u)\subseteq N(v)$, then $G$ is not $\gamma_{t}$-critical. \label{lem2}
\end{lem}


Mojdeh and Rad~\cite{MR} found the following lemma about a total domination vertex
critical graph $G$ of order $\Delta(G)+\gamma_{t}(G)$ with
$\delta(G)\geq 2$.

\begin{lem} [~\cite{MR}] There is no $3$-$\gamma_{t}$-critical graph
$G$ of order $\Delta(G)+3$ with $\Delta(G)=3$, $5$ and $\delta(G)\geq
2$. \label{lem4}\end{lem}


\section{Some properties of $\gamma_{t}$-critical graph $G$ with $\gamma_{t}(G)=n-\Delta(G)$} \label{intro}

In this section, we find some properties of
$\gamma_{t}$-critical graph $G$ with $\gamma_{t}(G)=n-\Delta(G)$.
Throughout the section, we assume the following notation
for $\gamma_{t}$-critical graph $G$ with $\gamma_{t}(G)=n-\Delta(G)$ and $\delta(G) \ge 2$
unless stated otherwise. Let $v$ be a vertex whose degree
is the maximum degree $\Delta(G)$. Since $G$ is $\gamma_{t}$-critical, it
follows that $\gamma_{t}(G-v)=\gamma_{t}(G)-1=n-\Delta(G)-1$. Let
$S$ be a $\gamma_{t}$-set of $G - v$. Then, $S = V(G)-N[v]$  by
Lemma~\ref{no-neighbor}.
Let $H_{1},H_{2},\cdots, H_{t}$ be the
components of $G[S]$ and let $V(H_i) = S_i$ for all
$i=1,2,\ldots,t$.
We find the following two lemmas regarding
the $\gamma_{t}$-critical graph $G$ with $\gamma_{t}(G)=n-\Delta(G)$ and $\delta(G) \ge 2$.

\begin{lem} \label{connected} Every $\gamma_{t}$-critical graph $G$ with
$\gamma_{t}(G)=n-\Delta(G)$ and $\delta(G) \ge 2$ is connected.
\begin{proof}
Suppose that $G$ is not connected. Then, at least one of
$H_{1}$, $H_{2}$, $\cdots$, $H_{t}$ is also a connected component of $G$,
say $H_i$ is such a component. Since $\delta(G) \ge 2$, $|V(H_i )|
=|S_i | \ge 3$. Choose a spanning tree $T$ of $H_i$ and one
end vertex $u$ of $T$. Then, $S_i - u$ is a total dominating set
of $H_i$ and furthermore $S- u$ is a total dominating set of
$G-v$, which is a contradiction.
\end{proof}
 \end{lem}

\begin{lem} \label{properties} Let $G$ be a $\gamma_{t}$-critical graph with
$\gamma_{t}(G)=n-\Delta(G)$ and $\delta(G) \ge 2$. Then,
\begin{enumerate}
 \item[$(1)$] $H_i$ is a $P_{2}$ or a $P_{3}$ for $i=1,2, \cdots,t$.
 \item[$(2)$] If $G[S]$ contains a $P_{3}$ component, then
 $G[S]=P_{3}$. Furthermore, for the $P_3 = u_1 u_2 u_3$, $N(u_2) \cap N(v) = \emptyset$ and $N(v)$ is a disjoint union of nonempty sets $N(u_1)-u_2$ and
 $N(u_3)-u_2$.
 \item[$(3)$] If $H_i$ is a $P_2$ for all $i=1,2, \ldots,t$, $i.e.,$ $H_i = u_iw_i$, then for any $u \in S$, $N(u) \cap N(v) \neq \emptyset$ and
 $N(v)$ is a disjoint union of $N(u_1)-w_1,N(w_1)-u_1, \ldots ,N(u_t)-w_t,N(w_t)-u_t $.  \\
 \end{enumerate}
\end{lem}
\begin{proof} (1) First, we aim to
show that $\Delta(H_i) \leq 2$ for $i=1,2,\cdots,t$. Suppose that
$\Delta(H_j) \geq 3$ for some $1\le j\le t$. Let $u$ be a
vertex of $H_j$ whose degree in $H_j$ is at least $3$. Choose a
spanning tree $T$ of $H_j$ containing all edges incident to $u$.
Then, $T$ has at least three leaves. Let $u_{1},u_{2},u_{3}$ be
three leaves in $T$. For any $x \in N(v)$, let
$S'=(S-\{u_{2},u_{3}\}) \cup \{v,x\}$. Then, $S'$ is a total
dominating set of $G$ and hence $\gamma_{t}(G) \leq
|S^{\prime}|=|S|=\gamma_{t}(G)-1,$ which is a contradiction.
Therefore, $\Delta(H_i) \leq 2$ for $i=1,2,\cdots,t$. It implies
that $H_i$ is a path or a cycle for $i=1,2,\cdots, t$.

Suppose that there exists $j$ such that $H_j$ is a cycle $u_{1}u_{2}\cdots
u_{k}u_1$ for $k \geq 3$. Then, there is $u_\ell$ such that
$N(u_\ell) \cap N(v) \neq \emptyset$. Without loss of generality,
we assume $N(u_1) \cap N(v) \neq \emptyset$ and pick a vertex $x \in
N(u_1) \cap N(v)$. Then, $S''=(S-\{u_{2},u_{3}\}) \cup \{v,x\}$ is
a total dominating set of $G,$ which is a contradiction. So, for
all $i = 1,2, \ldots, t$, $H_i$ is a path.

Suppose that there exists a path $H_i =u_{1}u_{2}\cdots u_{k}$ for
$k \geq 4$. Then, $(S-\{u_{1},u_{k}\}) \cup \{v,x\}$ for some $x
\in N(v)$ is a total dominating set of $G,$ which is a
contradiction. Therefore, $H_i$ is a $P_{2}$ or a $P_{3}$ for all
$i=1,2, \cdots,t$.

 \noindent (2) Let $G[S]$ contains a $P_{3}$
component, say $u_{1}u_{2}u_{3}$. If $G[S]$ contains another
component $w_{1}w_{2}w_{3}$ which is isomorphic to $P_{3}$, then for some $x \in N(v)$,
$(S-\{u_{3},w_{3}\}) \cup \{v,x\}$ is a total dominating set of
$G$, which is a contradiction. Next if we suppose $G[S]$ contains a $P_{3}$ and at least one $P_{2}$, say
$w_{1}w_{2}$. Then, $N(v) \cap N(w_1) \neq
\emptyset$ because $\delta(G) \ge 2$. For some $x \in N(v) \cap
N(w_1)$, $(S-\{u_{3},w_{2}\}) \cup \{v,x \}$ is a total dominating
set of $G$, it leads us a contradiction. Therefore, $G[S]=P_{3}=u_{1}u_{2}u_{3}$.

Since $\delta(G) \ge 2$, $(N(u_i)-u_2) \cap N(v) \neq \emptyset$
for any $i=1$ or $3$. If $N(u_2) \cap N(v) \neq \emptyset$ then
for any $x \in N(u_2) \cap N(v)$, $\{v , x , u_2 \}$ is a total
dominating set of $G$, which is a contradiction. Hence, $N(v)$ is
a disjoint union of nonempty sets $N(u_1)-u_2$ and
 $N(u_3)-u_2$.

\begin{figure}
\begin{pspicture}[shift=-1.2](-1.6,-2.2)(1.6,3.2)
\psline[linecolor=emgreen, linewidth=2pt](0,3)(-1,2)
\psline[linecolor=emgreen, linewidth=2pt](0,3)(1,2)
\psline[linecolor=emgreen, linewidth=2pt](-1,0)(-1,-1)
\psline(-1,-1)(0,-1.5)(1,-1)
\psline[linecolor=emgreen, linewidth=1.5pt](1,-1)(1,0)
\rput[t](-1,.8){$.$}
\rput[t](-1,.7){$.$}
\rput[t](-1,.6){$.$}
\rput[t](1,.8){$.$}
\rput[t](1,.7){$.$}
\rput[t](1,.6){$.$}
\rput[t](0,-2){$a)$}
\psframe[linecolor=darkred,linewidth=1pt](-1.5,0)(-.5,2)
\psframe[linecolor=darkred,linewidth=1pt](.5,0)(1.5,2)
\pscircle[fillstyle=solid,fillcolor=darkgray,linecolor=black](-1,1.7){.1}
\pscircle[fillstyle=solid,fillcolor=darkgray,linecolor=black](-1,1.4){.1}
\pscircle[fillstyle=solid,fillcolor=darkgray,linecolor=black](-1,1.1){.1}
\pscircle[fillstyle=solid,fillcolor=darkgray,linecolor=black](-1,.3){.1}
\pscircle[fillstyle=solid,fillcolor=darkgray,linecolor=black](1,1.7){.1}
\pscircle[fillstyle=solid,fillcolor=darkgray,linecolor=black](1,1.4){.1}
\pscircle[fillstyle=solid,fillcolor=darkgray,linecolor=black](1,1.1){.1}
\pscircle[fillstyle=solid,fillcolor=darkgray,linecolor=black](1,.3){.1}
\pscircle[fillstyle=solid,fillcolor=darkgray,linecolor=black](0,3){.1}
\pscircle[fillstyle=solid,fillcolor=darkgray,linecolor=black](0,-1.5){.1}
\pscircle[fillstyle=solid,fillcolor=darkgray,linecolor=black](-1,-1){.1}
\pscircle[fillstyle=solid,fillcolor=darkgray,linecolor=black](1,-1){.1}
\end{pspicture}
\begin{pspicture}[shift=-1.7](-6,-2.2)(6.5,4.2)
\psline[linecolor=emgreen, linewidth=2pt](0,4)(-1,2)
\psline[linecolor=emgreen, linewidth=2pt](0,4)(-3,2)
\psline[linecolor=emgreen, linewidth=2pt](0,4)(-5,2)
\psline[linecolor=emgreen, linewidth=2pt](0,4)(1,2)
\psline[linecolor=emgreen, linewidth=2pt](0,4)(3.5,2)
\psline[linecolor=emgreen, linewidth=2pt](0,4)(5.5,2)
\psline[linecolor=emgreen, linewidth=2pt](-5,0)(-5,-1) \psline(-5,-1)(-3,-1)
\psline[linecolor=emgreen, linewidth=2pt](-3,-1)(-3,0)
\psline[linecolor=emgreen, linewidth=2pt](-1,0)(-1,-1) \psline(-1,-1)(1,-1)
\psline[linecolor=emgreen, linewidth=2pt](1,-1)(1,0)
\psline[linecolor=emgreen, linewidth=2pt](3.5,0)(3.5,-1) \psline(3.5,-1)(5.5,-1)
\psline[linecolor=emgreen, linewidth=2pt](5.5,-1)(5.5,0)
\rput[t](2.2,1){$\ldots$}
\rput[t](-5,.8){$.$}
\rput[t](-5,.7){$.$}
\rput[t](-5,.6){$.$}
\rput[t](-3,.8){$.$}
\rput[t](-3,.7){$.$}
\rput[t](-3,.6){$.$}
\rput[t](-1,.8){$.$}
\rput[t](-1,.7){$.$}
\rput[t](-1,.6){$.$}
\rput[t](1,.8){$.$}
\rput[t](1,.7){$.$}
\rput[t](1,.6){$.$}
\rput[t](3.5,.8){$.$}
\rput[t](3.5,.7){$.$}
\rput[t](3.5,.6){$.$}
\rput[t](5.5,.8){$.$}
\rput[t](5.5,.7){$.$}
\rput[t](5.5,.6){$.$}
\rput[t](0,-1.5){$b)$}
\psframe[linecolor=darkred,linewidth=1pt](-5.5,0)(-4.5,2)
\psframe[linecolor=darkred,linewidth=1pt](-3.5,0)(-2.5,2)
\psframe[linecolor=darkred,linewidth=1pt](-1.5,0)(-.5,2)
\psframe[linecolor=darkred,linewidth=1pt](.5,0)(1.5,2)
\psframe[linecolor=darkred,linewidth=1pt](3,0)(4,2)
\psframe[linecolor=darkred,linewidth=1pt](5,0)(6,2)
\pscircle[fillstyle=solid,fillcolor=darkgray,linecolor=black](-5,1.7){.1}
\pscircle[fillstyle=solid,fillcolor=darkgray,linecolor=black](-5,1.4){.1}
\pscircle[fillstyle=solid,fillcolor=darkgray,linecolor=black](-5,1.1){.1}
\pscircle[fillstyle=solid,fillcolor=darkgray,linecolor=black](-5,.3){.1}
\pscircle[fillstyle=solid,fillcolor=darkgray,linecolor=black](-3,1.7){.1}
\pscircle[fillstyle=solid,fillcolor=darkgray,linecolor=black](-3,1.4){.1}
\pscircle[fillstyle=solid,fillcolor=darkgray,linecolor=black](-3,1.1){.1}
\pscircle[fillstyle=solid,fillcolor=darkgray,linecolor=black](-3,.3){.1}
\pscircle[fillstyle=solid,fillcolor=darkgray,linecolor=black](-1,1.7){.1}
\pscircle[fillstyle=solid,fillcolor=darkgray,linecolor=black](-1,1.4){.1}
\pscircle[fillstyle=solid,fillcolor=darkgray,linecolor=black](-1,1.1){.1}
\pscircle[fillstyle=solid,fillcolor=darkgray,linecolor=black](-1,.3){.1}
\pscircle[fillstyle=solid,fillcolor=darkgray,linecolor=black](1,1.7){.1}
\pscircle[fillstyle=solid,fillcolor=darkgray,linecolor=black](1,1.4){.1}
\pscircle[fillstyle=solid,fillcolor=darkgray,linecolor=black](1,1.1){.1}
\pscircle[fillstyle=solid,fillcolor=darkgray,linecolor=black](1,.3){.1}
\pscircle[fillstyle=solid,fillcolor=darkgray,linecolor=black](3.5,1.7){.1}
\pscircle[fillstyle=solid,fillcolor=darkgray,linecolor=black](3.5,1.4){.1}
\pscircle[fillstyle=solid,fillcolor=darkgray,linecolor=black](3.5,1.1){.1}
\pscircle[fillstyle=solid,fillcolor=darkgray,linecolor=black](3.5,.3){.1}
\pscircle[fillstyle=solid,fillcolor=darkgray,linecolor=black](5.5,1.7){.1}
\pscircle[fillstyle=solid,fillcolor=darkgray,linecolor=black](5.5,1.4){.1}
\pscircle[fillstyle=solid,fillcolor=darkgray,linecolor=black](5.5,1.1){.1}
\pscircle[fillstyle=solid,fillcolor=darkgray,linecolor=black](5.5,.3){.1}
\pscircle[fillstyle=solid,fillcolor=darkgray,linecolor=black](0,4){.1}
\pscircle[fillstyle=solid,fillcolor=darkgray,linecolor=black](-5,-1){.1}
\pscircle[fillstyle=solid,fillcolor=darkgray,linecolor=black](-3,-1){.1}
\pscircle[fillstyle=solid,fillcolor=darkgray,linecolor=black](-1,-1){.1}
\pscircle[fillstyle=solid,fillcolor=darkgray,linecolor=black](1,-1){.1}
\pscircle[fillstyle=solid,fillcolor=darkgray,linecolor=black](3.5,-1){.1}
\pscircle[fillstyle=solid,fillcolor=darkgray,linecolor=black](5.5,-1){.1}
\end{pspicture}
\caption{Figures of $\gamma_{t}$-critical graph $G$ with $\gamma_{t}(G)=n-\Delta(G)$ and $\delta(G) \ge 2$
 where all vertices in the boxes are adjacent
to vertices connected to boxes by thick lines and there could
be edges between vertices in different boxes or in the same box. This convention will be used for other figures.}
\label{fig0}
\end{figure}
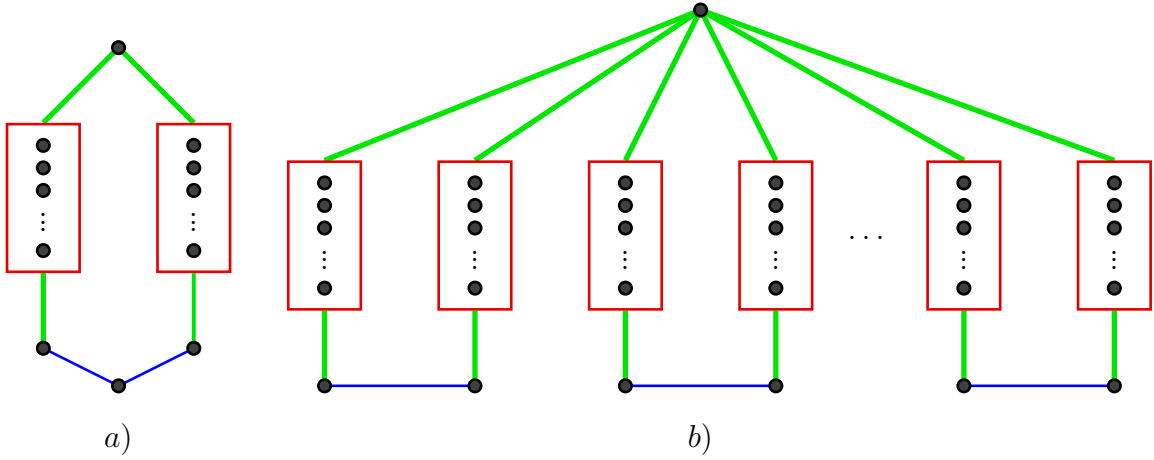

\noindent (3) Since $\delta(G) \ge 2$, $N(u) \cap N(v) \neq
\emptyset$ for any $u \in S$. Furthermore, for any $x \in N(v)$,
$N(x) \cap S \neq \emptyset$ because $S$ is a total dominating set
of $G-v$. We want to show that $|N(x) \cap S | = 1$ for any $ x
\in N(v)$. Suppose that there exists an $x \in N(v)$ such that
$u_i, w_i \in N(x)$ for some $i =1,2, \ldots,t $. Then, $S' =
(S-\{u_i,w_i \}) \cup \{v,x \}$ is a total dominating set of $G$,
which is a contradiction.

For the next case, suppose that there exists an $x \in N(v)$ such
that $u_i, u_j \in N(x)$ for some different $i,j$. Choose $y_i \in
N(v) \cap N(w_i)$ and $y_j \in  N(v) \cap N(w_j)$. Then, one can
easily check that $(S-\{u_i,u_j,w_i,w_j \}) \cup \{v,x,y_i,y_j \}$
is a total dominating set of $G$, which is a contradiction.
Similarly, one can show that a contradiction occurs if $|N(x) \cap
S| \ge 2$ for some $x \in N(v)$. It implies that $N(v)$ is a
disjoint union of $N(u_1)-w_1,N(w_1)-u_1, \ldots
,N(u_t)-w_t,N(w_t)-u_t $.
\end{proof}

These results can be summarized to obtain general figures of $\gamma_{t}$-critical
graph $G$ with $\gamma_{t}(G)=n-\Delta(G)$ and $\delta(G) \ge 2$ as
in Figure~\ref{fig0}.

\section{Proof of Theorem 2} \label{main}


In this section, we shall give a proof of Theorem~\ref{mainthm0}.
Suppose that $G$ is an $m$-$\gamma_t$-critical graph of order
$\Delta (G)+m$. Let $v$ be a vertex for which $d(v)=\Delta (G)$.
By Lemma~\ref{properties}, each connected component of
$G[V(G)-N[v]]$ is a $P_2$ or a $P_3$ and if there exists a component
$P_3$ then $G[V(G)-N[v]] = P_3$. Hence, $m-1 = |G[V(G)-N[v]]|$ is
3 or even. It implies that  $m=4$ or $m \ge 3$ is odd.

If $m=4$, then $G[V(G)-N[v]]$ is a $P_3$ and $\Delta(G) \ge 2$ by
Lemma~\ref{properties} (2). If $m \ge 3$ is odd then each component
of $G[V(G)-N[v]]$ is a $P_2$ and $\Delta(G) \ge m-1$ by
Lemma~\ref{properties} (3). Hence, for $m=4$ or odd $m \ge 3$ if
$\Delta < 2\lfloor\frac{m-1}{2}\rfloor$, then there exists no
$m$-$\gamma_t$-critical graph of order $\Delta + m$.

For $m=4$ and for even $\Delta \ge 2$, let $G$ be a graph whose vertex set
is $\{v\} \cup (U \cup W)\cup \{u_1,u_2,u_3 \}$ with $|U|=|W|=\Delta/2$
and whose edge set is composed of $\{vx, vy, u_1x, u_3y  | x \in U ,\  y \in W \} \cup\{u_1u_2, u_2u_3 \}$ as in Figure~\ref{fig0} a) and the subgraph induced by the vertices in between $U$ and $W$
 is $K_{\Delta/2, \Delta/2}-E(M)$, where $M$ is an 1-factor of $K_{\Delta/2, \Delta/2}$.
 Then, one can show that $G$ is a 4-$\gamma_t$-critical graph of order $\Delta (G)+4$.

For odd $m \ge 3$ and for even $\Delta \ge m-1$, let $G_1$ be a
graph  whose vertex set is $\{v_1\} \cup (U_1 \cup W_1) \cup
\{u_1,w_1 \}$ with $|U_1|=|W_1|=(\Delta-m+3)/2$ and whose edge set
is composed of $\{v_1x, v_1y, u_1x, w_1y  | x \in U_1 ,\  y \in
W_1 \} \cup\{u_1w_1 \}$ and the subgraph induced by the vertices in between $U_1$ and $W_1$
 is $K_{(\Delta-m+3)/2,
(\Delta-m+3)/2}-E(M)$, where $M$ is an 1-factor of $K_{(\Delta-m+3)/2,
(\Delta-m+3)/2}$. Then, one can show that $G_1$
is a $3$-$\gamma_t$-critical graph of order $\Delta (G_1)+3=
\Delta-m+6$. Note that $C_5$ is a $3$-$\gamma_t$-critical graph
of order $5$. So, the vertex amalgamation $G$ of $G_1$ and
$(m-3)/2$  5-cycles  with $v_1$ and any vertex in each $(m-3)/2$
5-cycles as in Figure~\ref{fig1} is an $m$-$\gamma_t$-critical graph of order $\Delta-m+6
+ 4\cdot \frac{m-3}{2}=\Delta+m$ by Proposition~\ref{vertex-amal}.
Hence, for $m=4$ or odd $m \ge 3$ and for any even $\Delta \ge
2\lfloor \frac{m-1}{2}\rfloor$, there exists an
$m$-$\gamma_t$-critical graph of order $\Delta (G)+m$.

\begin{figure}
\begin{pspicture}[shift=-.7](-4.5,-2.2)(4.5,1.2)
\psline(0,1)(-.5,-.5)(-.5,-1.5)(.5,-1.5)(.5,-.5)(0,1)
\psline(0,1)(3.5,-.5)(3.5,-1.5)(4.5,-1.5)(4.5,-.5)(0,1)
\psline[linecolor=emgreen, linewidth=2pt](0,1)(-2,0)
\psline[linecolor=emgreen, linewidth=2pt](0,1)(-4,0)
\psline[linecolor=emgreen, linewidth=2pt](-2,-1.5)(-2,-2)
\psline[linecolor=emgreen, linewidth=2pt](-4,-1.5)(-4,-2)
\psline(-4,-2)(-2,-2)
\rput[t](2,-1){$\ldots$}
\rput[t](-4,-.8){$.$}
\rput[t](-4,-.9){$.$}
\rput[t](-4,-1){$.$}
\rput[t](-2,-.8){$.$}
\rput[t](-2,-.9){$.$}
\rput[t](-2,-1){$.$}
\psframe[linecolor=darkred,linewidth=1pt](-4.5,-1.5)(-3.5,0)
\psframe[linecolor=darkred,linewidth=1pt](-2.5,-1.5)(-1.5,0)
\pscircle[fillstyle=solid,fillcolor=darkgray,linecolor=black](-4,-.2){.06}
\pscircle[fillstyle=solid,fillcolor=darkgray,linecolor=black](-4,-.4){.06}
\pscircle[fillstyle=solid,fillcolor=darkgray,linecolor=black](-4,-.6){.06}
\pscircle[fillstyle=solid,fillcolor=darkgray,linecolor=black](-4,-1.3){.06}
\pscircle[fillstyle=solid,fillcolor=darkgray,linecolor=black](-2,-.2){.06}
\pscircle[fillstyle=solid,fillcolor=darkgray,linecolor=black](-2,-.4){.06}
\pscircle[fillstyle=solid,fillcolor=darkgray,linecolor=black](-2,-.6){.06}
\pscircle[fillstyle=solid,fillcolor=darkgray,linecolor=black](-2,-1.3){.06}
\pscircle[fillstyle=solid,fillcolor=darkgray,linecolor=black](-4,-2){.1}
\pscircle[fillstyle=solid,fillcolor=darkgray,linecolor=black](-2,-2){.1}
\pscircle[fillstyle=solid,fillcolor=darkgray,linecolor=black](0,1){.1}
\pscircle[fillstyle=solid,fillcolor=darkgray,linecolor=black](-.5,-.5){.1}
\pscircle[fillstyle=solid,fillcolor=darkgray,linecolor=black](.5,-.5){.1}
\pscircle[fillstyle=solid,fillcolor=darkgray,linecolor=black](-.5,-1.5){.1}
\pscircle[fillstyle=solid,fillcolor=darkgray,linecolor=black](.5,-1.5){.1}
\pscircle[fillstyle=solid,fillcolor=darkgray,linecolor=black](3.5,-.5){.1}
\pscircle[fillstyle=solid,fillcolor=darkgray,linecolor=black](4.5,-.5){.1}
\pscircle[fillstyle=solid,fillcolor=darkgray,linecolor=black](3.5,-1.5){.1}
\pscircle[fillstyle=solid,fillcolor=darkgray,linecolor=black](4.5,-1.5){.1}
\end{pspicture}
\caption{Figures of $m$-$\gamma_t$-critical graph of order $\Delta+m$, where each box contains $\frac{\Delta-m+3}{2}$ vertices and the subgraph induced by the vertices in two boxes
 is $K_{(\Delta-m+3)/2,
(\Delta-m+3)/2}-E(M)$, where $M$ is a 1-factor of $K_{(\Delta-m+3)/2,
(\Delta-m+3)/2}$.}
\label{fig1}
\end{figure}
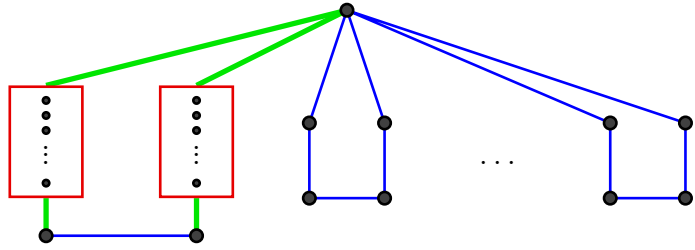

Now, we want to consider odd $\Delta$. In the paper~\cite{MR}, Mojdeh and Rad showed that there is no
$3$-$\gamma_{t}$-critical graph $G$ of order $ \Delta(G)+3$ for
$\Delta(G)= 3, 5$. In~\cite{CS}, Chen and Sohn proved
that there is no $3$-$\gamma_{t}$-critical graph of order
$\Delta(G)+3$ with $\Delta(G)=7$. Furthermore, they gave a family
of $3$-$\gamma_{t}$-critical graphs of order $\Delta(G)+3$ with
$\Delta(G)\geq 9$.  For any odd $m \ge 3$ and for any odd $\Delta \ge m+6$, let $G_2$ be a
$3$-$\gamma_{t}$-critical graph of order
$12$ with $\Delta(G_2)=9$ and $\delta(G_2) \ge 2$ and let $G_3$ be an
$(m-2)$-$\gamma_{t}$-critical graph of order
$\Delta +m-11$ with $\Delta(G_3)=\Delta-9 \ge m-3$ and $\delta(G_3) \ge 2$. Let $v_i \in V(G_i)$ be a vertex such that
$d(v_i)=\Delta(G_i)$ for each $i=2,3$. Then, the vertex amalgamation $G$ of $G_2$ and  $G_3$ with the vertices $v_2$ and $v_3$ is an $m$-$\gamma_{t}$-critical graph of order
$\Delta +m$ with $\Delta(G)=\Delta$ and $\delta(G_3) \ge 2$ by Proposition~\ref{vertex-amal}.
In the next section, we construct a $4$-$\gamma_t$-critical graph of
order $\Delta +4$  for any odd $\Delta \ge 9$. Hence, for any
$m=4$ or odd $m \ge 3$ and for any odd $\Delta \ge 2\lfloor
\frac{m-1}{2}\rfloor +7$, there exists an $m$-$\gamma_t$-critical
graph of order $\Delta + m$.

\section{$m=4$ or odd $m \ge 9$} \label{m4}

The only remaining open cases are $\Delta =
2\lceil\frac{m-1}{3}\rceil +k$, $k=1, 3, 5$. In this section, we
prove that there is no $4$-$\gamma_{t}$-critical graph of order
$\Delta +4$ with $\delta(G)\geq 2$ for $\Delta = 3,5$ or $7$. For
odd $m \ge 9$, it will  be shown that there exists an
$m$-$\gamma_{t}$-critical graph of order $\Delta +m$ with
$\delta(G)\geq 2$ for any odd $\Delta \ge
2\lceil\frac{m-1}{3}\rceil +1$.

\begin{thm}\label{mainthm3} There is no $4$-$\gamma_{t}$-critical graph
$G$ of order $\Delta(G)+4$ with $\Delta(G)= 3, 5, 7$ and
$\delta(G) \geq 2$.
\end{thm}

\begin{proof} Let $G$ be a $\gamma_{t}$-critical graph with $\gamma_{t}= n- \Delta(G)$ and $\delta(G)\geq 2$.
For any vertex $u\in V(G)$, let $S_{u}$ be a $\gamma_{t}(G-u)$-set. Choose $v\in V(G)$ such that $d(v)=\Delta(G)$.
Since $n(G) =\Delta(G)+4$, we can assume that $V(G)-N[v]=\{u,z,w\}$.
Since $G$ is $4$-$\gamma_{t}$-critical, by Lemma~\ref{lem1}, it follows
that $S_{v}=\{u,z,w\}$ and $N(u)\cup N(w)-\{z\}=N(v).$
Furthermore, $N(u)\cap N(w)=\{z\}$. Otherwise, say $x\in
N(u)\cap N(w)$, then $\{v,x,u\}$ is a $\gamma_{t}(G)$-set, which is
a contradiction.

Suppose that  $|N(u)\cap N(v)| \ge 2$ and $|N(w)\cap N(v)| \ge 2$. Then, for any $x \in N(u)\cap N(v)$, $S_x =\{ z,w,y \}$ or $\{ w, y, x_1 \}$ for some $y\in N(w)\cap N(v)$ and $x_1 \in N(u)\cap N(v)$. If $S_x =\{ z,w,y \}$ then $y$ dominates all elements in $ N(u)\cap N(v) - \{ x \}$ and hence, $\{ w, y, x_2 \}$ is also a total dominating set of $G-x$ for any $x_2 \in N(u)\cap N(v) - \{ x \}$. Therefore, we assume that for any $t \in N(v)$, $|S_t \cap N(v)| \ge 2$ in the case $|N(u)\cap N(v)| \ge 2$ and $|N(w)\cap N(v)| \ge 2$.

It divides into three cases depending on $\Delta(G)$.

\noindent\emph{Case 1}. $\Delta(G)=3$.  We assume that $N(u)\cap N(v) = \{x_1\}$ and $N(w)\cap N(v) = \{y_1, y_2\}$.
Since $G-y_2$ is  the cycle $C_{6}$ which has a total domination number $4$. It is a contradiction.

\noindent\emph{Case 2}. $\Delta(G)=5$. It divides into two cases depending on $|N(u)\cap N(v)|$.

\noindent\emph{Case 2.1}. We assume that
$N(u)\cap N(v) = \{x_1\}$ and $N(w)\cap N(v) = \{y_1, y_2, y_3,y_4\}$. It is obvious that there is no edges $x_1y_j$ ($j = 1,2,3,4$) in $G$.
If we delete $y_1$, there is the cycle $C_{6}$ in $G$ which have a total domination number $4$. It is a contradiction.

\noindent\emph{Case 2.2}. We assume that $N(u)\cap N(v) = \{x_1, x_2\}$ and $N(w)\cap N(v) = \{y_1, y_2, y_3\}$.
It is obvious that for any $i=1,2,3$, $y_i$ cannot be adjacent to both $x_1$ and $x_2$.
Without loss of generality, we can assume that  $x_1y_1, x_1y_2 \notin E(G)$.
It implies that  $S_{x_{2}} = \{ x_1, y_3, w \},$  $x_1y_3 \in E(G)$ and $x_2y_3 \notin E(G)$. By considering $S_{y_3}$, one can show that $x_2y_1\in E(G)$ or $x_2y_2 \in E(G)$. Let $x_2y_1 \in E(G)$. Then, $S_{y_1}=\{ x_1, y_3, u \}$ and $y_2y_3 \in E(G)$. Furthermore, $S_{y_2} = \{ x_2, y_1, u \}$ and $y_1y_3 \in E(G)$. In this case, $\{ x, y_3, u \}$ is a total dominating set of $G$, a contradiction.

\noindent\emph{Case 3}. $\Delta(G)=7$.  It divides into three cases depending on $|N(u)\cap N(v)|$.

\noindent\emph{Case 3.1}. We assume that $N(u)\cap N(v) = \{x_1\}$ and $N(w)\cap N(v) = \{y_1, y_2, y_3,y_4, y_5, y_6\}$.
It is obvious that $G$ is not $4$-$\gamma_{t}$-critical graph.

\noindent\emph{Case 3.2}. We assume that $N(u)\cap N(v) = \{x_1, x_2\}$ and $N(w)\cap N(v) = \{y_1, y_2, y_3, y_4, y_5\}$.
By the Pigeonhole Principle, we can assume that $x_{1}\in S_{y_{1}}\cap S_{y_{2}}\cap S_{y_{3}}$. For $j = 1,2,3$, $S_{y_{j}} \cap \{y_4, y_5\} \neq \emptyset$. By the Pigeonhole Principle, we can assume that  $S_{y_{1}} = S_{y_{2}} = \{ x_1, y_4, u\}$. Since $\{x_1, y_4, u\}$ is a $\gamma_{t}(G-y_1)$-set and $x_1y_2 \notin E(G)$, $y_2y_4 \in E(G).$ Therefore $\{ x_1, y_4, u\}$ is not a $\gamma_{t}(G-y_2)$-set. It is a contradiction.

\noindent\emph{Case 3.3}. We assume that $N(u)\cap N(v) = \{x_1, x_2, x_3\}$ and $N(w)\cap N(v) = \{y_1, y_2, y_3,y_4 \}$. It divides into four cases depending on existing edges between $\{x_1, x_2, x_3\}$. Suppose that there is no edges in  $\{x_1, x_2, x_3\}$.
Without loss of generality, let $S_{x_{1}} =\{ x_2, y_1, w\}$. Then,  $x_2y_1, x_3y_1 \in E(G)$ and $x_1y_1 \notin E(G)$. By the similar way, we can assume that $x_1y_2, x_3y_2 \in E(G)$ and  $x_1y_3, x_2y_3 \in E(G)$. Furthermore,  $x_2y_2, x_3y_3 \notin E(G)$.
Considering  $S_{y_{4}}$, we may assume that $S_{y_4} = \{x_1, y_2, u \}$. Then, $y_1y_2 \in E(G)$ and $x_1y_4, y_2y_4 \notin E(G)$. If $x_2y_4 \in E(G)$ or $x_3y_4 \in E(G)$ then $\{x_2, y_1, u\}$ or $\{x_3, y_1, u\}$ is a $\gamma_{t}(G)$-set, which is a contradiction. Hence, we may assume that
 $x_2y_4, x_3y_4 \notin E(G)$.  It implies that $S_{y_2}=\{x_2, y_3, u \}$ and hence $y_3y_4 \in E(G)$.  Let us consider $S_{y_3}$. Since $y_2y_4 \notin E(G)$, $S_{y_3} = \{ x_3, y_1, u \}$. It implies that $y_1y_4 \in E(G)$. Then, $\{x_2, y_1, u\}$ is a $\gamma_{t}(G)$-set,  a contradiction.

If there is one edges in  $\{x_1, x_2, x_3\}$,  we assume that $x_2x_3 \in E(G)$.  Without loss of generality, let $S_{x_{1}} = \{x_2, y_1, w \}$. Then,   $x_2y_1 \in E(G)$ and $x_1y_1 \notin E(G)$.  Also, without loss of generality, we may assume that  $S_{x_{2}} = \{x_1, y_2, w \}$. It implies that $x_1y_2 \in E(G)$ and $x_3y_2 \in E(G)$. In this case, $\{x_3, y_2, w\}$ is a $\gamma_{t}(G)$-set.  It is a contradiction.

If there is two or three edges in  $\{x_1, x_2, x_3\}$, one can similarly get a contradiction as the case that there is one edge in  $\{x_1, x_2, x_3\}$.
\end{proof}



\begin{lem} Let $G$ be a connected graph
  with $\Delta(G)=9$ or $\Delta(G)\geq 11$. Then there are positive integers  $3,2=s_{1},s_{2}= s_{3}$
  satisfying the following two conditions;
\begin{enumerate}
\item
 $ 3+ 2 +s_{2}+s_{3}=\Delta(G)$
\item
 $2 = s_{1}\leq s_{2} = s_{3}$.
\end{enumerate}\label{lem6}
\end{lem}
Now we  construct a family of $4$-$\gamma_{t}$-critical graphs  of
order $\Delta(G)+4$ with  $\delta(G)\geq 2$ and $\Delta(G)=9$ or
$\Delta(G)\geq 11$.

Let $H$ be a copy of the complement graph $\overline{K_{3}}$ of the complete graph $K_{3}$ . Let
$V(H)=\{x_1, x_2, x_3\}$. Let $H_{i}$ be a graph with
a vertex set  $V(H_{i})=\{y_{i1},y_{i2},\cdots,y_{is_{i}}\}$ for
$i=1,2,3$.  Suppose that  $2 = s_{1}\leq
s_{2}= s_{3}$.
Let $F$ be the graph obtained from $H_{1}\cup H_{2}\cup H_{3}$ by
adding edges $y_{1j}y_{2k}$, $y_{2k}y_{3\ell}$, $y_{1j}y_{3\ell}$ for
$j=1,2$, $k=1,2,\cdots,s_{2}$, and $\ell=1,2,\cdots,s_{3}$, $j\neq k$, $j\neq \ell$, $k\neq \ell$.
Let $G$ be the graph obtained from $H\cup F$  and four new vertices
$v,u,z,w$  by adding edges $x_i y_{jk}$ for $1\leq i,j\leq 3$,
$i\neq j$ and $1\leq k\leq s_{j}$, and then joining $v$ to every
vertex in $H\cup F$, joining $u$  to every vertex in $H$
and joining $w$ to every vertex in $F$, and adding the edges $uz$ and $zw$. Then $\Delta(G)=
3+2+s_{2}+s_{3}$. Two figures in Figure~\ref{fig2} are examples of $4$-$\gamma_{t}$-critical graphs with $\Delta(G)= 9, 11$.

\begin{figure}
\begin{pspicture}[shift=-1.7](-2.3,-2.2)(4.3,7.2)
\rput[t](0,0){$$}
\psline(2.5,2.5)(1.5,4.5)(2.5,.5)
\psline(2.5,4.5)(1.5,2.5)(2.5,.5)
\psline(2.5,2.5)(1.5,.5)(2.5,4.5)
\psline(-1,-1)(0,-2)(2,-1)
\psline(2.5,4.5)(-1,2.5)(2.5,.5)
\psline(1.5,4.5)(-1,2.5)(1.5,.5)
\psline(2.5,2.5)(-1,4)(2.5,.5)
\psline(1.5,2.5)(-1,4)(1.5,.5)
\psline(2.5,4.5)(-1,1)(2.5,2.5)
\psline(1.5,4.5)(-1,1)(1.5,2.5)
\psline[linecolor=emgreen, linewidth=2pt](0,7)(-1,5)
\psline[linecolor=emgreen, linewidth=2pt](0,7)(2,5)
\psline[linecolor=emgreen, linewidth=2pt](-1,-1)(-1,0)
\psline[linecolor=emgreen, linewidth=2pt](2,-1)(2,0)
\psframe[linecolor=darkred,linewidth=1pt](-1.5,0)(-.5,5)
\psframe[linecolor=darkred,linewidth=1pt](1,0)(3,5)
\pscircle[fillstyle=solid,fillcolor=darkgray,linecolor=black](0,7){.1}
\pscircle[fillstyle=solid,fillcolor=darkgray,linecolor=black](-1,4){.1}
\pscircle[fillstyle=solid,fillcolor=darkgray,linecolor=black](-1,2.5){.1}
\pscircle[fillstyle=solid,fillcolor=darkgray,linecolor=black](-1,1){.1}
\pscircle[fillstyle=solid,fillcolor=darkgray,linecolor=black](-1,-1){.1}
\pscircle[fillstyle=solid,fillcolor=darkgray,linecolor=black](1.5,4.5){.1}
\pscircle[fillstyle=solid,fillcolor=darkgray,linecolor=black](2.5,4.5){.1}
\pscircle[fillstyle=solid,fillcolor=darkgray,linecolor=black](1.5,2.5){.1}
\pscircle[fillstyle=solid,fillcolor=darkgray,linecolor=black](2.5,2.5){.1}
\pscircle[fillstyle=solid,fillcolor=darkgray,linecolor=black](1.5,.5){.1}
\pscircle[fillstyle=solid,fillcolor=darkgray,linecolor=black](2.5,.5){.1}
\pscircle[fillstyle=solid,fillcolor=darkgray,linecolor=black](0,-2){.1}
\pscircle[fillstyle=solid,fillcolor=darkgray,linecolor=black](2,-1){.1}
\end{pspicture}
\begin{pspicture}[shift=-1.7](-2.3,-2.2)(5.3,7.2)
\rput[t](0,0){$$}
\psline(3.3,2.5)(1.5,4.5)(3.5,.5)
\psline(3.5,4.5)(1.5,2.5)(3.5,.5)
\psline(3.3,2.5)(1.5,.5)(3.5,4.5)
\psline(-1,-1)(0,-2)(2.5,-1)
\psline(3.5,4.5)(-1,2.5)(3.5,.5)
\psline(1.5,4.5)(-1,2.5)(1.5,.5)
\psline(3.3,2.5)(-1,4)(3.5,.5)
\psline(1.5,2.5)(-1,4)(1.5,.5)
\psline(3.5,4.5)(-1,1)(3.3,2.5)
\psline(1.5,4.5)(-1,1)(1.5,2.5)
\psline(3.5,.5)(2.5,4.5)(1.5,.5)
\psline(3.5,4.5)(2.5,.5)(1.5,4.5)
\psline(3.5,.5)(3.3,2.5)(3.5,4.5)
\psline(2.5,4.5)(-1,2.5)(2.5,.5)
\psline(2.5,4.5)(1.5,2.5)(2.5,.5)
\psline(2.5,4.5)(-1,1)
\psline(2.5,.5)(-1,4)
\psline[linecolor=emgreen, linewidth=2pt](0,7)(-1,5)
\psline[linecolor=emgreen, linewidth=2pt](0,7)(2.5,5)
\psline[linecolor=emgreen, linewidth=2pt](-1,-1)(-1,0)
\psline[linecolor=emgreen, linewidth=2pt](2.5,-1)(2.5,0)
\psframe[linecolor=darkred,linewidth=1pt](-1.5,0)(-.5,5)
\psframe[linecolor=darkred,linewidth=1pt](1,0)(4,5)
\pscircle[fillstyle=solid,fillcolor=darkgray,linecolor=black](0,7){.1}
\pscircle[fillstyle=solid,fillcolor=darkgray,linecolor=black](-1,4){.1}
\pscircle[fillstyle=solid,fillcolor=darkgray,linecolor=black](-1,2.5){.1}
\pscircle[fillstyle=solid,fillcolor=darkgray,linecolor=black](-1,1){.1}
\pscircle[fillstyle=solid,fillcolor=darkgray,linecolor=black](-1,-1){.1}
\pscircle[fillstyle=solid,fillcolor=darkgray,linecolor=black](1.5,4.5){.1}
\pscircle[fillstyle=solid,fillcolor=darkgray,linecolor=black](2.5,4.5){.1}
\pscircle[fillstyle=solid,fillcolor=darkgray,linecolor=black](3.5,4.5){.1}
\pscircle[fillstyle=solid,fillcolor=darkgray,linecolor=black](1.5,2.5){.1}
\pscircle[fillstyle=solid,fillcolor=darkgray,linecolor=black](3.3,2.5){.1}
\pscircle[fillstyle=solid,fillcolor=darkgray,linecolor=black](1.5,.5){.1}
\pscircle[fillstyle=solid,fillcolor=darkgray,linecolor=black](2.5,.5){.1}
\pscircle[fillstyle=solid,fillcolor=darkgray,linecolor=black](3.5,.5){.1}
\pscircle[fillstyle=solid,fillcolor=darkgray,linecolor=black](0,-2){.1}
\pscircle[fillstyle=solid,fillcolor=darkgray,linecolor=black](2.5,-1){.1}
\end{pspicture}
\caption{Figures of $4$-$\gamma_{t}$-critical graphs with $\Delta(G)= 9, 11$.}
\label{fig2}
\end{figure}
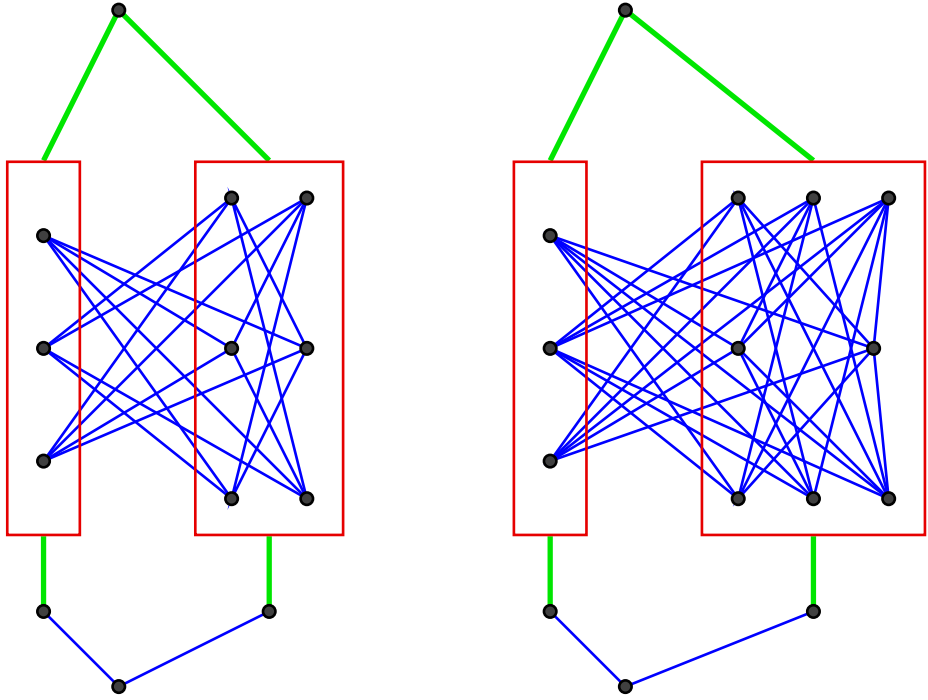

\begin{thm}  \label{mainthm4} The graph  $G$ in Figure \ref{fig2} is  $4$-$\gamma_{t}$-critical.\end{thm}
\begin{proof} It is obvious that $\gamma_{t}(G)=4$. So we only prove that $G$ is
$\gamma_{t}$-critical graph.  First,  $\{v,y_{11},w\}$,  $\{v, x_1, u\}$, $\{v, x_1, y_{11}\}$
and $\{u,w,z\}$ is a total dominating set of $G-u$, $G-w$, $G-z$ and $G-v$
respectively. For any vertex $x_i \in V(G)$, $\{w,y_{i1}, z\}$ is a
total dominating set of $G- x_i$. For any vertex $y_{jk}\in
V(G)$, It is easy to choose a total dominating set of $G-y_{jk}$,  In
general, for any vertex $a \in V(G)$, $\gamma_{t}(G-a)=3$. So $G$ is a $4$-$\gamma_{t}$-critical graph.
\end{proof}

By Theorems \ref{mainthm0}, \ref{mainthm3} and \ref{mainthm4}, we have the following corollary.

\begin{cor} \label{m=4summary}
There is a $4$-$\gamma_{t}$-critical graph $G$ of order $\Delta(G) +4$ with
$\delta(G) \geq 2$ if and only if  $\Delta(G) = 2,4,6,8$   or  $\Delta(G) \ge 9$.
\end{cor}

From now on, we aim to consider an $m$-$\gamma_{t}$-critical graph $G$ of order $\Delta(G) +m$ with
$\delta(G)\geq 2$ for any odd $m \ge 9$ and odd $\Delta(G) \ge
m$.

\begin{figure}
\begin{pspicture}[shift=-1.7](-4,-.2)(4.5,3.2)
\psline(0,3)(-.5,1)(-.5,0)(-1.5,0)(-1.5,1)(0,3)
\psline(0,3)(3.5,1)(3.5,0)(2.5,0)(2.5,1)(0,3)
\psline(0,3)(-2.5,1)(-2.5,0)(-3.5,0)(-3.5,1)(0,3)
\psline(0,3)(.3,.9) \psline(0,3)(.7,1.1) \psline(.7,1.1)(.5,0)
\psline(.3,.9)(.5,0)(1.5,0)(1.5,1)(0,3)
\pccurve[angleA=30,angleB=157](-3.5,1)(.7,1.1)
\pccurve[angleA=27,angleB=160](-2.5,1)(.7,1.1)
\pccurve[angleA=-30,angleB=-157](-1.5,1)(.3,.9)
\pccurve[angleA=-27,angleB=-160](-.5,1)(.3,.9)
\pccurve[angleA=-30,angleB=-157](1.5,1)(3.5,1)
\pccurve[angleA=-27,angleB=-160](1.5,1)(2.5,1)
\pscircle[fillstyle=solid,fillcolor=darkgray,linecolor=black](-3.5,0){.1}
\pscircle[fillstyle=solid,fillcolor=darkgray,linecolor=black](-3.5,1){.1}
\pscircle[fillstyle=solid,fillcolor=darkgray,linecolor=black](-2.5,1){.1}
\pscircle[fillstyle=solid,fillcolor=darkgray,linecolor=black](-2.5,0){.1}
\pscircle[fillstyle=solid,fillcolor=darkgray,linecolor=black](-1.5,1){.1}
\pscircle[fillstyle=solid,fillcolor=darkgray,linecolor=black](-1.5,0){.1}
\pscircle[fillstyle=solid,fillcolor=darkgray,linecolor=black](-.5,0){.1}
\pscircle[fillstyle=solid,fillcolor=darkgray,linecolor=black](-.5,1){.1}
\pscircle[fillstyle=solid,fillcolor=darkgray,linecolor=black](.5,0){.1}
\pscircle[fillstyle=solid,fillcolor=darkgray,linecolor=black](.3,.9){.1}
\pscircle[fillstyle=solid,fillcolor=darkgray,linecolor=black](.7,1.1){.1}
\pscircle[fillstyle=solid,fillcolor=darkgray,linecolor=black](0,3){.1}
\pscircle[fillstyle=solid,fillcolor=darkgray,linecolor=black](3.5,0){.1}
\pscircle[fillstyle=solid,fillcolor=darkgray,linecolor=black](3.5,1){.1}
\pscircle[fillstyle=solid,fillcolor=darkgray,linecolor=black](2.5,1){.1}
\pscircle[fillstyle=solid,fillcolor=darkgray,linecolor=black](2.5,0){.1}
\pscircle[fillstyle=solid,fillcolor=darkgray,linecolor=black](1.5,1){.1}
\pscircle[fillstyle=solid,fillcolor=darkgray,linecolor=black](1.5,0){.1}
\end{pspicture}
\caption{Figures of $m$ $\gamma_{t}$-critical graph $G$ with $\gamma_{t}(G)=n-\Delta(G)$ and $\delta(G) \ge 2$ for $m = 9$.}
\label{fig3}
\end{figure}
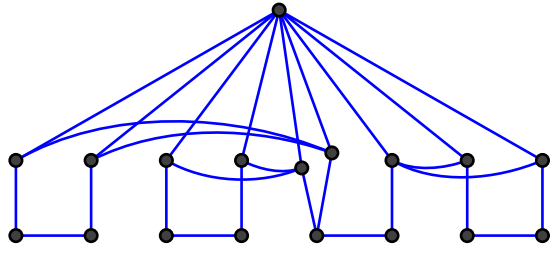

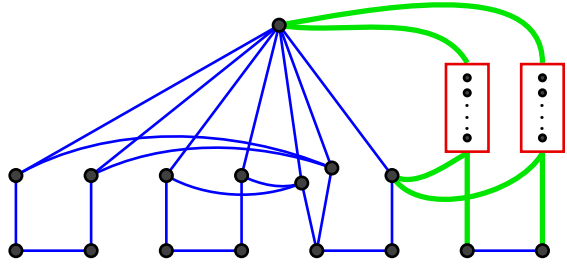
\begin{figure}
\begin{pspicture}[shift=-1.7](-4,-.2)(4.5,3.2)
\psline(0,3)(-.5,1)(-.5,0)(-1.5,0)(-1.5,1)(0,3)
\pccurve[angleA=10,angleB=90, linecolor=emgreen, linewidth=2pt](0,3)(3.5,2.5)
\psline[linecolor=emgreen, linewidth=2pt](3.5,1.3)(3.5,0)
\psline(3.5,0)(2.5,0)
\psline[linecolor=emgreen, linewidth=2pt](2.5,0)(2.5,1.3)
\pccurve[angleA=-10,angleB=120, linecolor=emgreen, linewidth=2pt](0,3)(2.5,2.5)
\psline(0,3)(-2.5,1)(-2.5,0)(-3.5,0)(-3.5,1)(0,3)
\psline(0,3)(.3,.9) \psline(0,3)(.7,1.1) \psline(.7,1.1)(.5,0)
\psline(.3,.9)(.5,0)(1.5,0)(1.5,1)(0,3)
\pccurve[angleA=30,angleB=157](-3.5,1)(.7,1.1)
\pccurve[angleA=27,angleB=160](-2.5,1)(.7,1.1)
\pccurve[angleA=-30,angleB=-157](-1.5,1)(.3,.9)
\pccurve[angleA=-27,angleB=-160](-.5,1)(.3,.9)
\pccurve[angleA=-60,angleB=-120, linecolor=emgreen, linewidth=2pt](1.5,1)(3.5,1.3)
\pccurve[angleA=-30,angleB=-150, linecolor=emgreen, linewidth=2pt](1.5,1)(2.5,1.3)
\rput[t](2.5,1.95){$.$}
\rput[t](2.5,1.8){$.$}
\rput[t](2.5,1.65){$.$}
\rput[t](3.5,1.95){$.$}
\rput[t](3.5,1.8){$.$}
\rput[t](3.5,1.65){$.$}
\psframe[linecolor=darkred,linewidth=1pt](2.2,1.3)(2.8,2.5)
\psframe[linecolor=darkred,linewidth=1pt](3.2,1.3)(3.8,2.5)
\pscircle[fillstyle=solid,fillcolor=darkgray,linecolor=black](-3.5,0){.1}
\pscircle[fillstyle=solid,fillcolor=darkgray,linecolor=black](-3.5,1){.1}
\pscircle[fillstyle=solid,fillcolor=darkgray,linecolor=black](-2.5,1){.1}
\pscircle[fillstyle=solid,fillcolor=darkgray,linecolor=black](-2.5,0){.1}
\pscircle[fillstyle=solid,fillcolor=darkgray,linecolor=black](-1.5,1){.1}
\pscircle[fillstyle=solid,fillcolor=darkgray,linecolor=black](-1.5,0){.1}
\pscircle[fillstyle=solid,fillcolor=darkgray,linecolor=black](-.5,0){.1}
\pscircle[fillstyle=solid,fillcolor=darkgray,linecolor=black](-.5,1){.1}
\pscircle[fillstyle=solid,fillcolor=darkgray,linecolor=black](.5,0){.1}
\pscircle[fillstyle=solid,fillcolor=darkgray,linecolor=black](.3,.9){.1}
\pscircle[fillstyle=solid,fillcolor=darkgray,linecolor=black](.7,1.1){.1}
\pscircle[fillstyle=solid,fillcolor=darkgray,linecolor=black](0,3){.1}
\pscircle[fillstyle=solid,fillcolor=darkgray,linecolor=black](3.5,0){.1}
\pscircle[fillstyle=solid,fillcolor=darkgray,linecolor=black](3.5,1.5){.06}
\pscircle[fillstyle=solid,fillcolor=darkgray,linecolor=black](3.5,2.3){.06}
\pscircle[fillstyle=solid,fillcolor=darkgray,linecolor=black](3.5,2.1){.06}
\pscircle[fillstyle=solid,fillcolor=darkgray,linecolor=black](2.5,1.5){.06}
\pscircle[fillstyle=solid,fillcolor=darkgray,linecolor=black](2.5,2.3){.06}
\pscircle[fillstyle=solid,fillcolor=darkgray,linecolor=black](2.5,2.1){.06}
\pscircle[fillstyle=solid,fillcolor=darkgray,linecolor=black](2.5,0){.1}
\pscircle[fillstyle=solid,fillcolor=darkgray,linecolor=black](1.5,1){.1}
\pscircle[fillstyle=solid,fillcolor=darkgray,linecolor=black](1.5,0){.1}
\end{pspicture}
\caption{Figures of $m$ $\gamma_{t}$-critical graph $G$ with $\gamma_{t}(G)=n-\Delta(G)$ and $\delta(G) \ge 2$ for $m \ge 9$,
where each box contains $\frac{\Delta-7}{2}$ vertices and the subgraph induced by the vertices in two boxes is
$K_{\frac{\Delta-7}{2},\frac{\Delta-7}{2}}-E(M)$, where $M$ is  a 1-factor of $K_{\frac{\Delta-7}{2},\frac{\Delta-7}{2}}$.}
\label{fig4}
\end{figure}

\begin{thm}  \label{mainthm5} For any odd $m \ge 9$ and for any odd $\Delta \ge
m$, there exists an $m$-$\gamma_{t}$-critical graph $G$ of order $\Delta +m$ with $\Delta(G)=\Delta$ and $\delta(G) \ge 2$.\end{thm}
\begin{proof}
Assume that there exists a $9$-$\gamma_{t}$-critical graph $G_1$ of order $\Delta_1 +9$ with $\Delta(G_1)=\Delta_1$ and $\delta(G_1) \ge 2$ for any odd $\Delta_1 \ge 9$. Then for odd $m \ge 9$ and for any odd $\Delta \ge
m$, one can construct $m$-$\gamma_{t}$-critical graph $G$ of order $\Delta +m$ with $\Delta(G)=\Delta$ and $\delta(G) \ge 2$ using a vertex amalgamation of $G_1$ and several $C_5$'s. Hence, it suffices to show that there exists a $9$-$\gamma_{t}$-critical graph $G$ of order $\Delta +9$ with $\Delta(G)=\Delta$ and $\delta(G) \ge 2$ for any odd $\Delta \ge 9$.

For any $\Delta \ge 9$, let $G=(V,E)$ be a graph whose vertex set is $\{v\} \cup \bigcup_{i=1}^{4}\left(U_i \cup W_i \cup \{u_i, w_i \} \right)$, where
 \begin{eqnarray*}
 U_i &=& \{ x_i \} ~\mbox{for $i=1,2$},~U_3=\{ x_{3 1}, x_{3 2} \}, U_4 = \{ x_{4 1}, x_{4 2}, \ldots, x_{4 \frac{\Delta-7}{2}} \}, \\
 W_i &=& \{ y_i \} ~\mbox{for $i=1,2,3$},~~ W_4 = \{ y_{4 1}, y_{4 2}, \ldots, y_{4 \frac{\Delta-7}{2}} \}
 \end{eqnarray*}
   and its edge set is composed of
   \begin{eqnarray*} & & \{ v x, v y, xu_i, yw_i, u_iw_i \ | \ x \in U_i ,\  y \in W_i,~i=1,2,3,4 \}  \\
& \cup & \{x_ix_{3 i}, y_ix_{3 i} \ | \ x_i \in U_i ,\  y_i \in W_i,~i=1,2 \}  \\
& \cup & \{y_3x, y_3y \ | \ y_3 \in W_3 ,\  x \in U_4 ,\ y \in W_4 \} \end{eqnarray*} as in Figure~\ref{fig3} and the subgraph induced by the vertices in $U_4$ and $W_4$
 is $K_{\frac{\Delta-7}{2},  \frac{\Delta-7}{2}}-E(M)$, where $M$ is a 1-factor of $K_{\frac{\Delta-7}{2},  \frac{\Delta-7}{2}}$.
 For our convenience, let $N_i = U_i \cup W_i \cup \{ u_i, w_i \}$ for $i=1,2,3,4$.  We want to show that $G$ is a 9-$\gamma_t$-critical graph of order $\Delta + 9$. Let $S$ be a total dominating set of $G$. Then, one can check that $\gamma_t(G) = |S| \ge 8$  because for $i=1,2,3,4$, $|S \cap N_i | \ge 2$ for $S$ to dominate $u_i$ and $w_j$. Suppose that $\gamma_t(G) = 8$. Then,  $|S \cap N_i | = 2$ for any $i=1,2,3,4$. Especially, $|S\cap N_3|=2$. If $S\cap N_3 = \{ x_{3 1}, u_3 \}$ then for $S$ to dominate $y_3$, $S \cap N_3$ is $\{x_{4 j}, u_4 \}$ or $\{y_{4 j}, w_4 \}$ for some $j=1,2, \ldots,\frac{\Delta-7}{2}$. In either cases, $W_4$ or $U_4$ is not dominated. For other choices of $S\cap N_3$, one can similarly show that $V(G)$ is not totally dominated by $S$ if $|S \cap N_3|=2$. So, $\gamma_t(G) = |S| \ge 9$. For $S_1 =\{u_i, w_i \ | \ i=1,2,4 \} \cup \{v, x_{3 1}, u_3 \}$, $S_1$ is total dominating set of $G$. Hence, $\gamma_t(G)=9$.

 If we delete  $u_j$ for some $j=1,2,3,4$, then for some $y \in W_j$, $\{u_i, w_i \ | \ i=1,2,3,4, ~i\neq j \} \cup \{ v,y \}$ is a total dominating set of $G-u_j$. Hence, $\gamma_t(G-u_j)=8$. Similarly, one can show that $\gamma_t(G-w_j)=8$. If we delete $x_1$ from $G$ then $\{u_i, w_i \ | \ i=2,3,4 \} \cup \{ y_1,w_1 \}$ is a total dominating set of $G-u_j$ and hence $\gamma_t(G-x_1)=8$. If we delete $x_{3,1}$ from $G$ then
 $\{u_1, w_1 , x_2, y_2, x_{3 2}, y_3, x_{4 1}, y_{4 1} \}$ is a total dominating set of $G-x_{3,1}$ and hence $\gamma_t(G-x_{3 1})=8$. Similarly, one can show that for any $z \in V(G)$, $\gamma_t(G-z)=8$. Therefore, $G$ is a 9-$\gamma_t$-critical graph of order $\Delta + 9$.
\end{proof}

By Theorems \ref{mainthm0} and \ref{mainthm5}, we have the following corollary.

\begin{cor} \label{m>=9summary}
For any odd $m \ge 9$, there exists an $m$-$\gamma_{t}$-critical graph $G$ of order $\Delta(G) +m$ with $\delta(G) \ge 2$ if and only if   $\Delta(G) \ge
m-1$.
\end{cor}

\noindent {\bf Remark:} We settled the
existence problem with respect to the parities of the total
domination number $m$ and the maximum degree $\Delta$ except some cases. The only remaining open cases are $ \Delta =5,7,9$ for $m=5$ and $\Delta =7,9,11$ for $m=7$.

\end{document}